\documentclass[reqno,english]{amsart}
\usepackage{amsfonts,amsmath,latexsym,verbatim,amscd,mathrsfs,color,array}
\usepackage[colorlinks=true]{hyperref}

\usepackage{amsmath,amssymb,amsthm,amsfonts,graphicx,color}
\usepackage{amssymb}
\usepackage{pdfsync}
\usepackage{epstopdf}

\usepackage{cite}

\usepackage{amsmath}
\usepackage{amsfonts}
\usepackage{amssymb}

 \usepackage{graphicx}

\newtheorem{theorem}{Theorem}

\newtheorem{definition}[theorem]{Definition}

\newtheorem{lemma}[theorem]{Lemma}

\newtheorem{proposition}[theorem]{Proposition}

\numberwithin{equation}{section}

\begin{document}

\title[elliptic sine-Gordon equation]{A complete classification of finite Morse index solutions to  elliptic
sine-Gordon equation}

\author[Y.Liu]{Yong Liu}

\address{\noindent School of Mathematics and Physics, North China Electric Power University, Beijing, China}
\email{liuyong@ncepu.edu.cn}

\author[J. Wei]{Juncheng Wei}
\address{\noindent
Department of Mathematics,
University of British Columbia, Vancouver, B.C., Canada, V6T 1Z2}
\email{jcwei@math.ubc.ca}

\begin{abstract}
The elliptic sine-Gordon equation in the plane has a family of explicit
multiple-end solutions (soliton-like solutions). We show that all the finite
Morse index solutions belong to this family. We also prove they are
non-degenerate in the sense that the corresponding linearized operators have
no nontrivial bounded kernel. We then show that solutions with $2n$ ends have Morse
index $n(n-1)/2.$

\end{abstract}

\maketitle

\section{Introduction and statement of the main results}

The classical sine-Gordon equation originally arises from the study of
surfaces with constant negative curvature in the nineteenth century. It also
appears in many physical contexts such as Josephson junction. It is an
important partial differential equation and has been extensively studied partly due to the fact that it
is an integrable system and one can use the technique of inverse scattering
transform to analyze its solutions. There are vast literatures on this
subject. We refer to the papers in the book \cite{Sine} and the references
cited there for more information about the background and detailed discussion
for this equation.

In the space-time coordinate, the sine-Gordon equation has the form
\begin{equation}
\partial_{x}^{2}u-\partial_{t}^{2}u=\sin u. \label{hs}%
\end{equation}
In this paper, we are interested in the elliptic version of this equation.
More precisely, we shall consider the problem%
\begin{equation}
-\Delta u=\sin u\ \text{ in }\mathbb{R}^{2},\text{ }\left\vert u\right\vert
<\pi. \label{SG}%
\end{equation}
This is a special case of the Allen-Cahn type equations
\begin{equation}
\Delta u=W^{\prime}\left(  u\right)  \ \text{ in }\mathbb{R}^{N}, \label{De}%
\end{equation}
where $W$ are double well potentials. The equation $\left(  \ref{SG}\right)  $
corresponds to $W \left(  u\right)  =1+\cos u.$ Note that if $W\left(
u\right)  =\frac{1}{4}\left(  u^{2}-1\right)  ^{2},$ then the corresponding
equation is the classical Allen-Cahn equation
\begin{equation}
-\Delta u = u-u^3 \ \text{ in }\mathbb{R}^{N},\text{ }\left\vert u\right\vert
<1. \label{AC}
\end{equation}

 Equation $\left(
\ref{SG}\right)  $ has a one dimensional \textquotedblleft
heteroclinic\textquotedblright\ solution
\[
H\left(  x\right)  =4\arctan e^{x}-\pi.
\]
It is monotone increasing. Translating and rotating it in the plane, we obtain
a family of one dimensional solutions. The celebrated De Giorgi conjecture
concerns the classification of monotone bounded solutions of the Allen-Cahn
type equation $\left(  \ref{De}\right)  $. Many works have been done towards a
complete resolution of this conjecture. We refer to
\cite{Ambro,M2,FMV,F1,F2,Gui,LWW,J,Savin} and the references therein for results in
this direction. We will study in this paper non-monotone solutions of $\left(
\ref{SG}\right)  $ in the plane.

Let us recall the following

\begin{definition} (\cite{Michal})
A solution $u$ of $\left(  \ref{SG}\right)  $ is called a $2n$-ended
solution, if outside a large ball, the nodal set of $u$ is asymptotic to
$2n$  half straight lines.
\end{definition}

These asymptotic lines are called ends of the solutions.  One can show that
actually along these half straight lines, the solution $u$ behaves like one
dimensional solution $H$ in the transverse direction. (See \cite{Michal}.) The set of $2n$-ended solution is denoted as  ${\mathcal M}_{2n}$. By a recent result of
Wang-Wei\cite{Wang}, a solution to $\left(  \ref{SG}\right)  $ is multiple-end
if and only if it is finite Morse index. In \cite{Manuel}, the infinite
dimensional Lyapunov-Schmidt reduction method has been used to construct a
family of $2n$-end solutions for the Allen-Cahn equation (\ref{AC}). The method can also
be applied to general double well potentials, including the elliptic
sine-Gordon equation$\left(  \ref{SG}\right)  $. The nodal sets of these
solutions are almost parallel, meaning that the angles between consecutive
ends are close to $0$ or $\pi.$ Actually, the nodal curves are given
approximately by a rescaled solutions of the Toda system. It is also known
that locally around each $2n$-end solution, the moduli space of $2n$-end
solutions has the structure of a real analytic variety. If the solution
happens to be nondegenerate, then locally around this solution, the moduli
space is indeed a $2n$-dimensional manifold\cite{Michal}. For general
nonlinearities, little is known for the structure of the moduli space of
$2n$-end solutions, except in the $n=2$ case. We now know\cite{K0,K1,K2} that
the space of $4$-end solutions is diffeomorphic to the open interval $\left(
0,1\right)  ,$ modulo translation and rotation(they give $3$-free parameters
in the moduli space). Based on these four-end solutions, an end-to-end
construction for $2n$-end solutions has been carried out in \cite{K3}. Roughly
speaking, these solutions are in certain sense lying near the boundary of the
moduli space.

The classification of ${\mathcal M}_{2n}$ is largely open for general nonlinearities.  Important open question include: are solutions in ${\mathcal M}_{2n}$ nondegenerate? Is ${\mathcal M}_{2n}$ connected? What is Morse index in $ {\mathcal M}_{2n}$? In a recent paper Mantoulidis \cite{Man}, a lower bound $ n-1$ on Morse index of ${\mathcal M}_{2n}$ is given. In this paper we give a complete answer to the above question in the case of the special elliptic sine-Gordon equation (\ref{SG}).

It is well known that the classical sine-Gordon equation $\left(
\ref{hs}\right)  $ is an integrable system. Methods from the theory of
integrable systems can be used to find solutions of this system. In
particular, it has soliton solutions. Note that $\left(  \ref{SG}\right)  $ is
elliptic, while $\left(  \ref{hs}\right)  $ is hyperbolic in nature. We find
in this paper that the Hirota direct method of integrable systems also gives
us real nonsingular solutions of $\left(  \ref{SG}\right)  .$ Let $U_{n}$ be
the functions defined by $\left(  \ref{Un}\right)  .$ Then $U_{n}-\pi$ are
solutions to $\left(  \ref{SG}\right)  ,$ they depends on $2n$ parameters,
$p_{j},\eta_{j}^{0},j=1,...,n.$ We are interested in the spectrum property of
these solutions. In this paper, we shall use B\"{a}cklund transformation to
prove the following

\begin{theorem}
\label{Main}The $2n$-end solutions $U_{n}-\pi$ of elliptic sine-Gordon
equation$\ \left(  \ref{SG}\right)  $ are $L^{\infty}$-nondegenerated in the
following sense: If $\phi$ is a bounded solution of the linearized equation
\[
-\Delta\phi+\phi\cos U_{n}=0.
\]
Then there exist constants $d_{j},j=1,...,n$, such that
\[
\phi=\sum_{j=1}^{n}\left(  d_{j}\partial_{\eta_{j}^{0}}U_{n}\right)  .
\]

\end{theorem}

We remark that the nonlinear stability of 2-soliton solutions of the classical
hyperbolic sine-Gordon equation $\left(  \ref{hs}\right)  $ has been proved
recently by Munoz-Palacios \cite{Munoz}, also using B\"{a}cklund transformation. We refer
to the references therein for more discussion on the dynamical properties of
the hyperbolic sine-Gordon equation. For general background and applications
of B\"{a}cklund transformation, we refer to \cite{Rogers,Rogers82}.

The Morse index of $U_{n}-\pi$ is by definition the number of negative
eigenvalues of the operator $-\Delta-\cos U_{n},$ in the space $H^{2}\left(
\mathbb{R}^{2}\right)  ,$ counted with multiplicity. The Morse index can also
be defined as the maximal dimension of the subspace of compactly supported
smooth functions where the associated quadratic form of the energy functional
is negative. Our next result is

\begin{theorem}
\label{Main2}The Morse index of $U_{n}-\pi$ is equal to $\frac{n\left(
n-1\right)  }{2}.$ Moreover, all the finite Morse index solutions of $\left(
\ref{SG}\right)  $ are of the form $U_{n}-\pi,$ with suitable choice of the
parameters $p_{j},q_{j},\eta_{j}^{0}.$
\end{theorem}

The classification result stated in this theorem follows from a direct
application of the inverse scattering transform studied in \cite{Gut}. Inverse
scattering transform for elliptic sine-Gordon equation has also been used in
\cite{Bori1,Borisov} to study solutions with periodic asymptotic behavior or
vortex type singularities. Note that certain class of vortex type solutions
were analyzed through B\"{a}cklund transformation or direct method in
\cite{Hudak,Leibbrandt,Nakamura,Takeno}, and finite energy solutions with
point-like singularities have been studied in \cite{Sha}. It is also worth
mentioning that more recently, some classes of quite involved boundary value
problems of the elliptic sine-Gordon equation have been investigated via Fokas
direct method in \cite{P2,P4,P1,P3}.

Theorem \ref{Main2} implies that in the special case $n=2,$ the four-end
solutions of the equation $\left(  \ref{SG}\right)  $ have Morse index one. In
the family of four-end solutions, there is a special one, called saddle
solution(see $\left(  \ref{saddle}\right)  $), explicitly given by%
\[
4\arctan\left(  \frac{\cosh\left(  \frac{y}{\sqrt{2}}\right)  }{\cosh\left(
\frac{x}{\sqrt{2}}\right)  }\right)  -\pi.
\]
The nodal set of this solution consists of two orthogonally intersected
straight lines, hence the name saddle solution. Saddle-shaped solutions of
Allen-Cahn type equation $\Delta u=F^{\prime}\left(  u\right)  $ in
$\mathbb{R}^{2k}$ with $k\geq2$ has been studied by Cabre and Terra in a
series of papers. In \cite{Ca1,Ca2} it is proved that in $\mathbb{R}^{4}$ and
$\mathbb{R}^{6},$ the saddle-shaped solution is unstable, while in
$\mathbb{R}^{2k}$ with $k\geq7,$ they are stable\cite{Ca3}. It is also
conjectured in \cite{Ca3} that for $k\geq4,$ the saddle-shaped solutions
should be a global minimizer of the energy functional. However, for $F\left(
u\right)  =1+\cos u,$ the generalized elliptic sine-Gordon equation in even
dimension higher than two is believed to be non-integrable, hence no explicit
formulas are available for these saddle-shaped solutions. We expect that our
nondegeneracy results in this paper will be useful in the construction of
solutions of the generalized elliptic sine-Gordon equation in higher dimensions.

We also stress that $W\left(  u\right)  =1+\cos u$ is essentially the only
double well potential such that the corresponding equation is
integrable\cite{Dodd}. It is also worth pointing out that the sine
nonlinearity also appears in the Pierls-Nabarro equation whose solutions have been classified in \cite{Toland}. In dimension two, a
classification result like Theorem \ref{Main2} for general double well
potentials could be very difficult.

Finally we mention that recently there have been many interesting  studies on the use of Allen-Cahn type equation in constructing minimal surfaces. We refer to \cite{CM}, \cite{Ga1}, \cite{Ga2}, \cite{Gua}, \cite{Man}  and the references therein.

$\mathbf{Acknowledgement}$ The research of J. Wei is partially supported by
NSERC of Canada. Part of this work was finished while the first author was
visiting the University of British Columbia in 2017. He thanks the institute
for the financial support.

\section{Soliton solutions of the hyperbolic sine-Gordon equation}

In this section, we consider the classical sine-Gordon equation
\begin{equation}
\partial_{x}^{2}\phi-\partial_{z}^{2}\phi=\sin\phi. \label{csg}%
\end{equation}
Hereafter, we shall call it hyperbolic sine-Gordon equation. It is well known
that this equation has soliton solutions. Let us first recall the explicit
$n$-soliton solutions of $\left(  \ref{csg}\right)  $ in the form obtained in
\cite{Hirota} using Hirota direct method. We also refer to
\cite{Gibbon,Hirota2,Hirota3,Waz} for related results on soliton solutions.

Let $P_{j},Q_{j}$ be complex numbers with $P_{j}^{2}-Q_{j}^{2}=1.$ Define%
\begin{equation}
\alpha\left(  j,k\right)  =\frac{\left(  P_{j}-P_{k}\right)  ^{2}-\left(
Q_{j}-Q_{k}\right)  ^{2}}{\left(  P_{j}+P_{k}\right)  ^{2}-\left(  Q_{j}%
+Q_{k}\right)  ^{2}}. \label{alf}%
\end{equation}
Note that $\alpha\left(  j,k\right)  =\alpha\left(  k,j\right)  .$ Since
\[
P_{j}-Q_{j}=\frac{1}{P_{j}+Q_{j}},
\]
we can also rewrite $\alpha$ in the form
\begin{align*}
\alpha\left(  j,k\right)   &  =\frac{\left(  P_{j}-P_{k}+Q_{j}-Q_{k}\right)
\left(  \frac{1}{P_{j}+Q_{j}}-\frac{1}{P_{k}+Q_{k}}\right)  }{\left(
P_{j}+P_{k}+Q_{j}+Q_{k}\right)  \left(  \frac{1}{P_{j}+Q_{j}}+\frac{1}%
{P_{k}+Q_{k}}\right)  }\\
&  =-\frac{\left(  P_{j}-P_{k}+Q_{j}-Q_{k}\right)  ^{2}}{\left(  P_{j}%
+P_{k}+Q_{j}+Q_{k}\right)  ^{2}}.
\end{align*}
We also define $a$ by
\[
a\left(  i_{1},i_{2},...,i_{n}\right)  =1,\text{ if }n=0\text{,}1,
\]%
\[
a\left(  i_{1},i_{2},...,i_{n}\right)  =\prod\limits_{k<l}\alpha\left(
i_{k},i_{l}\right)  ,\text{ if }n\geq2.
\]
Let us introduce the notation $\eta_{i}=P_{i}x-Q_{i}z-\eta_{i}^{0},$ where
$\eta_{i}^{0}$ is a complex constant.

It has been proved in \cite{Hirota} that equation $\left(  \ref{csg}\right)  $
has families of $n$-soliton solutions of the form
\begin{equation}
\phi=4\arctan\frac{g}{f}, \label{Hirota}%
\end{equation}
where the functions $f,g$ are explicitly given by
\begin{equation}
f=\sum_{k=0}^{\left[  n/2\right]  }\left(  \sum\limits_{\left\{  n,2k\right\}
}\left[  a\left(  i_{1},...,i_{2k}\right)  \exp\left(  \eta_{i_{1}}%
+...+\eta_{i_{2k}}\right)  \right]  \right)  , \label{f}%
\end{equation}%
\begin{equation}
g=\sum_{k=0}^{\left[  \left(  n-1\right)  /2\right]  }\left(  \sum
\limits_{\left\{  n,2k+1\right\}  }\left[  a\left(  i_{1},...,i_{2k+1}\right)
\exp\left(  \eta_{i_{1}}+...+\eta_{i_{2k+1}}\right)  \right]  \right)  .
\label{g}%
\end{equation}
Here the notation $\sum\limits_{\left\{  n,k\right\}  }$ means taking sum over
all possible $k$ different integers $i_{1},...,i_{k}$ from the set of integers
$\left\{  1,...,n\right\}  .$ It is worth mentioning that these solutions can
also be written in the Wronskian form(\cite{Nimmo}). Here we choose to use the
form $\left(  \ref{f}\right)  ,\left(  \ref{g}\right)  $, because it is more
convenient to check the positive condition of the function. This will be clear
when we are dealing with the elliptic version of the sine-Gordon equation.
Note that in the special case $n=3,$ we have
\begin{align*}
f  &  =\sum_{n=0}^{1}\left(  \sum\limits_{\left\{  3,2n\right\}  }a\left(
i_{1},...,i_{2n}\right)  \exp\left(  \eta_{i_{1}}+...+\eta_{i_{2n}}\right)
\right) \\
&  =1+a\left(  1,2\right)  \exp\left(  \eta_{1}+\eta_{2}\right)  +a\left(
1,3\right)  \exp\left(  \eta_{1}+\eta_{3}\right)  +a\left(  2,3\right)
\exp\left(  \eta_{2}+\eta_{3}\right) \\
&  =1+\alpha\left(  1,2\right)  \exp\left(  \eta_{1}+\eta_{2}\right)
+\alpha\left(  1,3\right)  \exp\left(  \eta_{1}+\eta_{3}\right)
+\alpha\left(  2,3\right)  \exp\left(  \eta_{2}+\eta_{3}\right)  ,
\end{align*}%
\begin{align*}
g  &  =\sum_{k=0}^{1}\left(  \sum\limits_{\left\{  3,2k+1\right\}  }a\left(
j_{1},...,j_{2k+1}\right)  \exp\left(  \eta_{i_{1}}+...+\eta_{i_{2k+1}%
}\right)  \right) \\
&  =\exp\left(  \eta_{1}\right)  +\exp\left(  \eta_{2}\right)  +\exp\left(
\eta_{3}\right)  +a\left(  1,2,3\right)  \exp\left(  \eta_{1}+\eta_{2}%
+\eta_{3}\right) \\
&  =\exp\left(  \eta_{1}\right)  +\exp\left(  \eta_{2}\right)  +\exp\left(
\eta_{3}\right)  +\alpha\left(  1,2\right)  \alpha\left(  1,3\right)
\alpha\left(  2,3\right)  \exp\left(  \eta_{1}+\eta_{2}+\eta_{3}\right)  .
\end{align*}

\subsection{B\"{a}cklund transform and bilinear form of the hyperbolic
sine-Gordon equation}

Lamb\cite{Lamb} has established a superposition formula for the B\"{a}cklund
transformation of the hyperbolic sine-Gordon equation. In particular, this
formula enables us to get multi-soliton solutions in an algebraic way.
However, in this formulation, for $n$-soliton solutions with $n$ large, it
will be quite tedious to write down the explicit expressions for the
solutions. Nevertheless, it turns out that the soliton solutions $\left(
\ref{Hirota}\right)  $ discussed above can be obtained through B\"{a}cklund
transformation. This will be discussed in more details in this section.

In the light-cone coordinate, the hyperbolic sine-Gordon equation has the
form
\begin{equation}
u_{st}=\sin u,\left(  s,t\right)  \in\mathbb{R}^{2}. \label{sg}%
\end{equation}
Let $\beta$ be a parameter. The B\"{a}cklund transformation between two
solutions $f$ and $g$ of $\left(  \ref{sg}\right)  $ is given by(see for
instance \cite{Rogers82}):
\begin{equation}
\left\{
\begin{array}
[c]{l}%
\partial_{s}f=\partial_{s}g+2\beta\sin\frac{f+g}{2},\\
\partial_{t}f=-\partial_{t}g+2\beta^{-1}\sin\frac{f-g}{2}.
\end{array}
\right.  \label{Backlund}%
\end{equation}

Next we recall the bilinear form of the hyperbolic sine-Gordon
equation(\cite{Hirota4}). Let $F=f+ig$ be a complex function, where $i$ is the
complex unit and $f,g$ are real valued functions. The complex conjugate of $F$
will be denoted by $\bar{F}.$ Now we write $u$ as
\[
u=2i\ln\frac{\bar{F}}{F}=4\arctan\frac{g}{f}.
\]
Note that
\[
\sin u=\frac{e^{iu}-e^{-iu}}{2i}=\frac{1}{2i}\left(  \frac{F^{2}}{\bar{F}^{2}%
}-\frac{\bar{F}^{2}}{F^{2}}\right)  .
\]
We use $D$ to denote the bilinear operator(\cite{Hirota4}). Then $\left(
\ref{sg}\right)  $ has the bilinear form
\[
D_{s}D_{t}F\cdot F=\frac{1}{2}\left(  F^{2}-\bar{F}^{2}\right)  .
\]
The B\"{a}cklund transformation can be written in the following bilinear
form(see \cite{Hirota4}):%
\begin{equation}
\left\{
\begin{array}
[c]{c}%
D_{s}G\cdot F=\frac{h}{2}\bar{G}\bar{F},\\
D_{t}G\cdot\bar{F}=\frac{1}{2h}\bar{G}F.
\end{array}
\right.  \label{bilinear}%
\end{equation}
If $u=2i\ln\frac{\bar{F}}{F},v=2i\ln\frac{\bar{G}}{G}$ satisfy $\left(
\ref{bilinear}\right)  ,$ then they also satisfy $\left(  \ref{Backlund}%
\right)  .$

Let us fix an $n\in\mathbb{N}.$ The $n$-soliton solutions discussed in the
previous section are indeed B\"{a}cklund transformation of certain $n-1$
soliton type solutions. To see this, we write the solutions in another form.
For $j=1,...,n,$ let $k_{j}=P_{j}+Q_{j}$ and define $\xi_{j}$ by
\[
e^{\xi_{j}}=\prod\limits_{l<j}\frac{k_{l}+k_{j}}{k_{l}-k_{j}}\prod
\limits_{l>j}\frac{k_{j}+k_{l}}{k_{j}-k_{l}}.
\]
Since $P_{j}^{2}-Q_{j}^{2}=1,$ we know that $k_{j}^{-1}=P_{j}-Q_{j}.$ Let us
now define $\tilde{\eta}_{j}=\eta_{j}-\xi_{j},j=1,...,n.$ It can be written as
$\tilde{\eta}_{j}=P_{j}x+Q_{j}z+\tilde{\eta}_{j,0},$ with $\tilde{\eta}%
_{j,0}=\eta_{j,0}-\xi_{j}.$

With these notations, the function $f_{n}$ can be rewritten as
\begin{align*}
&  \sum_{k=0}^{\left[  n/2\right]  }\left(  \sum\limits_{\left\{
n,2k\right\}  }a\left(  i_{1},...,i_{2k}\right)  \exp\left(  \xi_{i_{1}%
}+...+\xi_{i_{2k}}\right)  \exp\left(  \tilde{\eta}_{i_{1}}+...+\tilde{\eta
}_{i_{2k}}\right)  \right) \\
&  =\exp\left(  \frac{1}{2}\left(  \tilde{\eta}_{1}+...+\tilde{\eta}%
_{n}\right)  \right)  \tilde{f}_{n}\prod\limits_{l<j\leq n}\frac{1}%
{k_{l}-k_{j}},
\end{align*}
where the function $\tilde{f}_{n}$ is defined to be
\begin{equation}
\sum\limits_{\prod\limits_{k=1}^{n}\varepsilon_{k}=\left(  -1\right)  ^{n}%
}\left(  \exp\left(  \sum\limits_{j=1}^{n}\frac{\varepsilon_{j}}{2}\left(
\tilde{\eta}_{j}+\frac{\pi i}{2}\right)  +\frac{n\pi i}{4}\right)
\prod\limits_{m<j\leq n}\left(  k_{m}-\varepsilon_{m}\varepsilon_{j}%
k_{j}\right)  \right)  , \label{tiltaf}%
\end{equation}
and $\varepsilon_{j}$ are indices equal $+1$ or $-1.$ Similarly, we can write
\[
g_{n}=\exp\left(  \frac{1}{2}\left(  \tilde{\eta}_{1}+...+\tilde{\eta}%
_{n}\right)  \right)  \tilde{g}_{n}\prod\limits_{l<j\leq n}\frac{1}%
{k_{l}-k_{j}},
\]
with%
\begin{equation}
\tilde{g}_{n}=\sum\limits_{\prod\limits_{k=1}^{n}\varepsilon_{k}=\left(
-1\right)  ^{n+1}}\left(  \exp\left(  \sum\limits_{j=1}^{n}\frac
{\varepsilon_{j}}{2}\left(  \tilde{\eta}_{j}+\frac{\pi i}{2}\right)
+\frac{\left(  n-2\right)  \pi i}{4}\right)  \prod\limits_{m<j\leq n}\left(
k_{m}-\varepsilon_{m}\varepsilon_{j}k_{j}\right)  \right)  . \label{tiltag}%
\end{equation}
We see that the $n$-soliton solution $\left(  \ref{Hirota}\right)  $ of the
hyperbolic sine-Gordon equation also equals $4\arctan\frac{\tilde{g}_{n}%
}{\tilde{f}_{n}}.$

We next would like to consider an $n-1$-soliton solutions closely related to
$4\arctan\frac{\tilde{g}_{n}}{\tilde{f}_{n}}.$ More precisely, we define
\[
\gamma=\sum\limits_{\prod\limits_{k=1}^{n-1}\varepsilon_{k}=\left(  -1\right)
^{n-1}}\left(  \exp\left(  \sum\limits_{j=1}^{n-1}\frac{\varepsilon_{j}}%
{2}\left(  \tilde{\eta}_{j}+\frac{\pi i}{2}\right)  +\frac{\left(  n-1\right)
\pi i}{4}\right)  \prod\limits_{m<j\leq n-1}\left(  k_{m}-\varepsilon
_{m}\varepsilon_{j}k_{j}\right)  \right)  ,
\]%
\[
\tau=\sum\limits_{\prod\limits_{k=1}^{n-1}\varepsilon_{k}=\left(  -1\right)
^{n}}\left(  \exp\left(  \sum\limits_{j=1}^{n-1}\frac{\varepsilon_{j}}%
{2}\left(  \tilde{\eta}_{j}+\frac{\pi i}{2}\right)  +\frac{\left(  n-3\right)
\pi i}{4}\right)  \prod\limits_{m<j\leq n-1}\left(  k_{m}-\varepsilon
_{m}\varepsilon_{j}k_{j}\right)  \right)  .
\]
Let $x=s+t,z=s-t.$ We have following B\"{a}cklund transformation.

\begin{lemma}
\label{Back}Let $F=\gamma+i\tau,$ $G=\tilde{f}_{n}+i\tilde{g}_{n}.$ Suppose
$P_{j},Q_{j},\tilde{\eta}_{j,0},j=1,...,n,$ are real numbers. Then
\[
\left\{
\begin{array}
[c]{c}%
D_{s}G\cdot F=-\frac{1}{2k_{n}}\bar{G}\bar{F},\\
D_{t}G\cdot\bar{F}=-\frac{k_{n}}{2}\bar{G}F.
\end{array}
\right.
\]

\end{lemma}

\begin{proof}
We sketch the proof of this fact for completeness. We only prove the first
identity, the idea for the proof of the second one is similar.

We compute
\begin{align*}
D_{s}G\cdot F  &  =F\partial_{s}G-G\partial_{s}F\\
&  =\left(  \partial_{s}\tilde{f}_{n}+i\partial_{s}\tilde{g}_{n}\right)
\left(  \gamma+\tau i\right)  -\left(  \tilde{f}_{n}+i\tilde{g}_{n}\right)
\left(  \partial_{s}\gamma+i\partial_{s}\tau\right) \\
&  =\gamma\partial_{s}\tilde{f}_{n}-\tau\partial_{s}\tilde{g}_{n}-\left(
\tilde{f}_{n}\partial_{s}\gamma-\tilde{g}_{n}\partial_{s}\tau\right) \\
&  +\left[  \left(  \tau\partial_{s}\tilde{f}_{n}+\gamma\partial_{s}\tilde
{g}_{n}\right)  -\left(  \tilde{f}_{n}\partial_{s}\tau+\tilde{g}_{n}%
\partial_{s}\gamma\right)  \right]  i.
\end{align*}
On the other hand,
\begin{align*}
\bar{G}\bar{F}  &  =\left(  \tilde{f}_{n}-i\tilde{g}_{n}\right)  \left(
\gamma-\tau i\right) \\
&  =\tilde{f}_{n}\gamma-\tilde{g}_{n}\tau-i\left(  \tilde{f}_{n}\tau+\tilde
{g}_{n}\gamma\right)  .
\end{align*}

We proceed to show that the real part of $D_{s}G\cdot F+\frac{1}{2k_{n}}GF$ is
equal to zero, that is
\[
\gamma\partial_{s}\tilde{f}_{n}-\tau\partial_{s}\tilde{g}_{n}-\left(
\tilde{f}_{n}\partial_{s}\gamma-\tilde{g}_{n}\partial_{s}\tau\right)
+\frac{1}{2k_{n}}\left(  \tilde{f}_{n}\gamma-\tilde{g}_{n}\tau\right)  =0.
\]
To see this, we first consider those terms involving $\exp\left(  \frac{1}%
{2}\varepsilon_{n}\tilde{\eta}_{n}\right)  $ with $\varepsilon_{n}=-1.$
Consider a typical sum $I$ of terms in $\gamma\partial_{s}\tilde{f}_{n},$ of
the form
\begin{align*}
&  \exp\left(  \sum\limits_{j=1}^{n-1}\frac{\varepsilon_{j}}{2}\left(
\tilde{\eta}_{j}+\frac{\pi i}{2}\right)  +\frac{\left(  n-1\right)  \pi i}%
{4}\right)  \prod\limits_{m<j\leq n-1}\left(  k_{m}-\varepsilon_{m}%
\varepsilon_{j}k_{j}\right) \\
&  \exp\left(  \sum\limits_{j=1}^{n}\frac{\hat{\varepsilon}_{j}}{2}\left(
\tilde{\eta}_{j}+\frac{\pi i}{2}\right)  +\frac{n\pi i}{4}\right)
\prod\limits_{m<j\leq n}\left(  k_{m}-\hat{\varepsilon}_{m}\hat{\varepsilon
}_{j}k_{j}\right)  \times\frac{1}{2}\left(  \hat{\varepsilon}_{1}k_{1}%
^{-1}+...+\hat{\varepsilon}_{n}k_{n}^{-1}\right) \\
&  +\exp\left(  \sum\limits_{j=1}^{n-1}\frac{\hat{\varepsilon}_{j}}{2}\left(
\tilde{\eta}_{j}+\frac{\pi i}{2}\right)  +\frac{\left(  n-1\right)  \pi i}%
{4}\right)  \prod\limits_{m<j\leq n-1}\left(  k_{m}-\hat{\varepsilon}_{m}%
\hat{\varepsilon}_{j}k_{j}\right)  \times\\
&  \exp\left(  \sum\limits_{j=1}^{n}\frac{\varepsilon_{j}}{2}\left(
\tilde{\eta}_{j}+\frac{\pi i}{2}\right)  +\frac{n\pi i}{4}\right)
\prod\limits_{m<j\leq n}\left(  k_{m}-\varepsilon_{m}\varepsilon_{j}%
k_{j}\right)  \times\frac{1}{2}\left(  \varepsilon_{1}k_{1}^{-1}%
+...+\varepsilon_{n}k_{n}^{-1}\right)  ,
\end{align*}
with $\prod\limits_{j=1}^{n-1}\varepsilon_{j}=\prod\limits_{j=1}^{n-1}%
\hat{\varepsilon}_{j}=\left(  -1\right)  ^{n-1},$ and $\varepsilon_{n}%
=\hat{\varepsilon}_{n}=-1.$ The function $I$ has a \textquotedblleft
dual\textquotedblright\ part $I^{\ast}$ in $\tilde{f}_{n}\partial_{s}\gamma,$
of the form
\begin{align*}
&  \exp\left(  \sum\limits_{j=1}^{n}\frac{\varepsilon_{j}}{2}\left(
\tilde{\eta}_{j}+\frac{\pi i}{2}\right)  +\frac{n\pi i}{4}\right)
\prod\limits_{m<j\leq n}\left(  k_{m}-\varepsilon_{m}\varepsilon_{j}%
k_{j}\right)  \times\\
&  \exp\left(  \sum\limits_{j=1}^{n-1}\frac{\hat{\varepsilon}_{j}}{2}\left(
\tilde{\eta}_{j}+\frac{\pi i}{2}\right)  +\frac{\left(  n-1\right)  \pi i}%
{4}\right)  \prod\limits_{m<j\leq n-1}\left(  k_{m}-\hat{\varepsilon}_{m}%
\hat{\varepsilon}_{j}k_{j}\right)  \times\frac{1}{2}\left(  \hat{\varepsilon
}_{1}k_{1}^{-1}+...+\hat{\varepsilon}_{n-1}k_{n-1}^{-1}\right) \\
&  +\exp\left(  \sum\limits_{j=1}^{n}\frac{\hat{\varepsilon}_{j}}{2}\left(
\tilde{\eta}_{j}+\frac{\pi i}{2}\right)  +\frac{n\pi i}{4}\right)
\prod\limits_{m<j\leq n}\left(  k_{m}-\hat{\varepsilon}_{m}\hat{\varepsilon
}_{j}k_{j}\right)  \times\\
&  \exp\left(  \sum\limits_{j=1}^{n-1}\frac{\varepsilon_{j}}{2}\left(
\tilde{\eta}_{j}+\frac{\pi i}{2}\right)  +\frac{\left(  n-1\right)  \pi i}%
{4}\right)  \prod\limits_{m<j\leq n-1}\left(  k_{m}-\varepsilon_{m}%
\varepsilon_{j}k_{j}\right)  \times\frac{1}{2}\left(  \varepsilon_{1}%
k_{1}^{-1}+...+\varepsilon_{n-1}k_{n-1}^{-1}\right)
\end{align*}
Subtracting $I$ with $I^{\ast},$ we obtain
\begin{align*}
&  \exp\left(  \sum\limits_{j=1}^{n}\frac{\varepsilon_{j}}{2}\left(
\tilde{\eta}_{j}+\frac{\pi i}{2}\right)  +\frac{n\pi i}{4}\right)
\prod\limits_{m<j\leq n}\left(  k_{m}-\varepsilon_{m}\varepsilon_{j}%
k_{j}\right)  \times\\
&  \exp\left(  \sum\limits_{j=1}^{n-1}\frac{\hat{\varepsilon}_{j}}{2}\left(
\tilde{\eta}_{j}+\frac{\pi i}{2}\right)  +\frac{\left(  n-1\right)  \pi i}%
{4}\right)  \prod\limits_{m<j\leq n-1}\left(  k_{m}-\hat{\varepsilon}_{m}%
\hat{\varepsilon}_{j}k_{j}\right)  \times\\
&  \frac{\varepsilon_{n}k_{n}^{-1}}{2}\left(  \prod\limits_{i<n}\left(
k_{i}-\varepsilon_{i}\varepsilon_{n}k_{n}\right)  +\prod\limits_{i<n}\left(
k_{i}-\hat{\varepsilon}_{i}\hat{\varepsilon}_{n}k_{n}\right)  \right)  .
\end{align*}
This corresponds to the sum of two terms in $-\frac{k_{n}^{-1}}{2}\tilde
{f}_{n}\gamma.$ Hence if one only considers those terms involving $\exp\left(
\frac{1}{2}\varepsilon_{n}\tilde{\eta}_{n}\right)  $ with $\varepsilon
_{n}=-1,$ then $\gamma\partial_{s}\tilde{f}_{n}-\tau\partial_{s}\tilde{g}%
_{n}=-\frac{k_{n}^{-1}}{2}\tilde{f}_{n}\gamma,$ similarly for $\tau
\partial_{s}\tilde{g}_{n}-\tilde{g}_{n}\partial_{s}\tau+\frac{k_{n}^{-1}}%
{2}\tilde{g}_{n}\tau.$

For those terms involving $\exp\left(  \frac{1}{2}\varepsilon_{n}\tilde{\eta
}_{n}\right)  $ with $\varepsilon_{n}=1,$ there is a similar cancelation
between $\gamma\partial_{s}\tilde{f}_{n}-\tau\partial_{s}\tilde{g}_{n}$ and
$-\frac{k_{n}^{-1}}{2}\tilde{g}_{n}\tau$, also there is cancelation between
$\tau\partial_{s}\tilde{g}_{n}-\tilde{g}_{n}\partial_{s}\tau$ and
$-\frac{k_{n}^{-1}}{2}\tilde{f}_{n}\gamma$.

Summarizing, we get
\[
\gamma\partial_{s}\tilde{f}_{n}-\tau\partial_{s}\tilde{g}_{n}-\left(
\tilde{f}_{n}\partial_{s}\gamma-\tilde{g}_{n}\partial_{s}\tau\right)
+\frac{k_{n}^{-1}}{2}\left(  \tilde{f}_{n}\gamma-\tilde{g}_{n}\tau\right)
=0.
\]
The proof is completed.
\end{proof}

\section{Multiple-end solutions and B\"{a}cklund transformation of the
elliptic sine-Gordon equation}

In this section, we consider the elliptic sine-Gordon equation in the form
\begin{equation}
\partial_{x}^{2}u+\partial_{y}^{2}u=\sin u. \label{esg}%
\end{equation}
Note that $u$ is a solution to $\left(  \ref{esg}\right)  $ if and only if
$u-\pi$ is a solution to $\left(  \ref{SG}\right)  .$ The elliptic sine-Gordon
equation has been studied by Leibbrandt in \cite{Leibbrandt}, with an
application to the Josephson effect. He uses the B\"{a}cklund transformation
method. However, the solutions he found is singular at some points in the
plane. Gutshabash-Lipovski\u{\i}\cite{Gut} studied the boundary value problem
of the elliptic sine-Gordon equation in the half plane using inverse
scattering transform and obtained mutli-soliton solutions in the determinant
form. The boundary problems have also been studied in \cite{P2,P4,P1,P3} by
the Fokas direct method.

Our observation in this paper is that in the hyperbolic sine-Gordon equation
$\left(  \ref{csg}\right)  ,$ it we introduce the changing of variable $z=yi,$
where $i$ is the complex unit, then we get the elliptic sine-Gordon equation.
Based on this, by choosing certain complex parameters in $\left(
\ref{f}\right)  ,\left(  \ref{g}\right)  $ for the solutions of the hyperbolic
sine-Gordon equation, we then get multiple-end solutions of the elliptic
sine-Gordon equation. Let us describe this in more details.

Let $p_{j},q_{j}$ be real numbers with $p_{j}^{2}+q_{j}^{2}=1.$ Similar as the
hyperbolic sine-Gordon case, we define%
\[
\alpha\left(  j,k\right)  =\frac{\left(  p_{j}-p_{k}\right)  ^{2}+\left(
q_{j}-q_{k}\right)  ^{2}}{\left(  p_{j}+p_{k}\right)  ^{2}+\left(  q_{j}%
+q_{k}\right)  ^{2}}.
\]
We still use the notation
\[
a\left(  i_{1},i_{2},...,i_{n}\right)  =1,\text{ if }n=0\text{,}1,
\]%
\[
a\left(  i_{1},i_{2},...,i_{n}\right)  =\prod\limits_{k<l}\alpha\left(
i_{k},i_{l}\right)  ,\text{ if }n\geq2.
\]
Define $\eta_{i}=p_{i}x-q_{i}y-\eta_{i}^{0}.$ Then the elliptic sine-Gordon
equation has the solution
\begin{equation}
U_{n}:=4\arctan\frac{g}{f}, \label{Un}%
\end{equation}
where
\[
f=\sum_{k=0}^{\left[  n/2\right]  }\left(  \sum\limits_{\left\{  n,2k\right\}
}\left[  a\left(  i_{1},...,i_{2k}\right)  \exp\left(  \eta_{i_{1}}%
+...+\eta_{i_{2k}}\right)  \right]  \right)  ,
\]%
\[
g=\sum_{m=0}^{\left[  \left(  n-1\right)  /2\right]  }\left(  \sum
\limits_{\left\{  n,2m+1\right\}  }\left[  a\left(  i_{1},...,i_{2m+1}\right)
\exp\left(  \eta_{i_{1}}+...+\eta_{i_{2m+1}}\right)  \right]  \right)  .
\]
Note that $U_{n}-\pi$ is indeed a smooth $2n$-end solution of $\left(
\ref{SG}\right)  .$

In the special case of $n=2,$ if we choose $p_{1}=p_{2}=p$ and $q_{1}%
=-q_{2}=q,$ $\eta_{1}^{0}=\eta_{2}^{0}=\ln\frac{p}{q},$ then we get the
solution
\[
\varphi_{p,q}\left(  x,y\right)  :=4\arctan\left(  \frac{p\cosh\left(
qy\right)  }{q\cosh\left(  px\right)  }\right)  -\pi.
\]
This corresponds to a four-end solution of the elliptic sine-Gordon equation
$\left(  \ref{SG}\right)  $. Note that on the lines $px=\pm qy,$
$\varphi_{p,q}=4\arctan\frac{p}{q}-\pi.$ In the special case $p=q=\frac
{\sqrt{2}}{2},$ the solution is
\begin{equation}
4\arctan\left(  \frac{\cosh\left(  \frac{y}{\sqrt{2}}\right)  }{\cosh\left(
\frac{x}{\sqrt{2}}\right)  }\right)  -\pi. \label{saddle}%
\end{equation}
This is the classical saddle solution.

We remark that this family of 4-end solutions has analogous in the minimal
surface theory. They are the so called Scherk second surface family, which are
embedded singly periodic minimal surfaces in $\mathbb{R}^{3}.$ Explicitly,
these surface can be described by
\[
\cos^{2}\theta\cosh\frac{x}{\cos\theta}-\sin^{2}\theta\sinh\frac{y}{\sin
\theta}=\cos z.
\]
Here $\theta$ is a parameter. Each of these surfaces has four wings, called
ends of the surfaces. Geometrically, they are obtained by desingularized two
intersected planes with intersection angle $\theta.$

Next, we would like to investigate the B\"{a}cklund transformation for the
solutions of elliptic sine-Gordon equation. Let $k_{j}=p_{j}+q_{j}i.$ Define
$\xi_{j}$ by
\[
e^{\xi_{j}}=\prod\limits_{l<j}\frac{k_{l}+k_{j}}{k_{l}-k_{j}}\prod
\limits_{j<l}\frac{k_{j}+k_{l}}{k_{j}-k_{l}}.
\]
Recall that for all $j,$ $p_{j}^{2}+q_{j}^{2}=1.$\ Hence the number
$\frac{k_{l}+k_{j}}{k_{l}-k_{j}}$ is purely imaginary and $e^{\xi_{j}}$ is in
general complex valued. We define $\tilde{\eta}_{j}=\eta_{j}-\xi_{j}%
=p_{j}x+q_{j}y+\eta_{j}^{0}-\xi_{j},j=1,...,n.$ Then the solution $U_{n}$ can
be written as $4\arctan\frac{\tilde{g}_{n}}{\tilde{f}_{n}},$ where
\[
\tilde{f}_{n}=\sum\limits_{\prod\limits_{j=1}^{n}\varepsilon_{j}=\left(
-1\right)  ^{n}}\left(  \exp\left(  \sum\limits_{j=1}^{n}\frac{\varepsilon
_{j}}{2}\left(  \tilde{\eta}_{j}+\frac{\pi i}{2}\right)  +\frac{n\pi i}%
{4}\right)  \prod\limits_{m<j\leq n}\left(  k_{m}-\varepsilon_{m}%
\varepsilon_{j}k_{j}\right)  \right)  ,
\]%
\[
\tilde{g}_{n}=\sum\limits_{\prod\limits_{j=1}^{n}\varepsilon_{j}=\left(
-1\right)  ^{n+1}}\left(  \exp\left(  \sum\limits_{j=1}^{n}\frac
{\varepsilon_{j}}{2}\left(  \tilde{\eta}_{j}+\frac{\pi i}{2}\right)
+\frac{\left(  n-2\right)  \pi i}{4}\right)  \prod\limits_{m<j\leq n}\left(
k_{m}-\varepsilon_{m}\varepsilon_{j}k_{j}\right)  \right)  .
\]
Furthermore, we define
\[
\gamma=\sum\limits_{\prod\limits_{j=1}^{n-1}\varepsilon_{j}=\left(  -1\right)
^{n-1}}\left(  \exp\left(  \sum\limits_{j=1}^{n-1}\frac{\varepsilon_{j}}%
{2}\left(  \tilde{\eta}_{j}+\frac{\pi i}{2}\right)  +\frac{\left(  n-1\right)
\pi i}{4}\right)  \prod\limits_{m<j\leq n-1}\left(  k_{m}-\varepsilon
_{m}\varepsilon_{j}k_{j}\right)  \right)  ,
\]%
\[
\tau=\sum\limits_{\prod\limits_{j=1}^{n-1}\varepsilon_{j}=\left(  -1\right)
^{n}}\left(  \exp\left(  \sum\limits_{j=1}^{n-1}\frac{\varepsilon_{j}}%
{2}\left(  \tilde{\eta}_{j}+\frac{\pi i}{2}\right)  +\frac{\left(  n-3\right)
\pi i}{4}\right)  \prod\limits_{m<j\leq n-1}\left(  k_{m}-\varepsilon
_{m}\varepsilon_{j}k_{j}\right)  \right)  .
\]
Let $x=s+t,y=-i\left(  s-t\right)  $ and $u=U_{n},$ $v=4\arctan\frac{\tau
}{\gamma}.$ A direct consequence of Lemma \ref{Back} is the following

\begin{lemma}
\bigskip The functions $u$ and $v$ are connected through the following
B\"{a}cklund transformation:
\begin{equation}
\left\{
\begin{array}
[c]{l}%
\partial_{x}u=-i\partial_{y}v+\bar{k}_{n}\sin\frac{u+v}{2}+k_{n}\sin\frac
{u-v}{2},\\
i\partial_{y}u=-\partial_{x}v-\bar{k}_{n}\sin\frac{u+v}{2}+k_{n}\sin\frac
{u-v}{2}.
\end{array}
\right.  \label{uv}%
\end{equation}

\end{lemma}

We remark that $\frac{\tau}{\gamma}$ is purely imaginary. The function
$\sin\frac{v}{2}$ is understood to be
\begin{align*}
\sin\left(  2\arctan\frac{\tau}{\gamma}\right)   &  =\frac{2\gamma\tau}%
{\gamma^{2}+\tau^{2}},\\
\cos\left(  2\arctan\frac{\tau}{\gamma}\right)   &  =\frac{\gamma^{2}-\tau
^{2}}{\gamma^{2}+\tau^{2}},
\end{align*}
and $\partial_{x}v=4\frac{\gamma\partial_{x}\tau-\tau\partial_{x}\gamma
}{\gamma^{2}+\tau^{2}}.$

\section{\bigskip Linearized B\"{a}cklund transformation and nondegeneracy of
the $2n$-end solution of the elliptic sine-Gordon equation}

Linearizing the B\"{a}cklund transformation $\left(  \ref{uv}\right)  $ at
$\left(  u,v\right)  ,$ we get the linearized system
\[
\left\{
\begin{array}
[c]{l}%
\partial_{x}\phi=-i\partial_{y}\eta+\bar{k}_{n}\cos\frac{u+v}{2}\left(
\frac{\phi+\eta}{2}\right)  +k_{n}\cos\frac{u-v}{2}\left(  \frac{\phi-\eta}%
{2}\right)  ,\\
i\partial_{y}\phi=-\partial_{x}\eta-\bar{k}_{n}\cos\frac{u+v}{2}\left(
\frac{\phi+\eta}{2}\right)  +k_{n}\cos\frac{u-v}{2}\left(  \frac{\phi-\eta}%
{2}\right)  .
\end{array}
\right.
\]
It can be written as%
\begin{equation}
\left\{
\begin{array}
[c]{c}%
L\phi=M\eta,\\
T\phi=N\eta,
\end{array}
\right.  \label{LB}%
\end{equation}
where
\begin{align*}
L\phi &  =\partial_{x}\phi-\left(  \bar{k}_{n}\cos\frac{u+v}{2}+k_{n}\cos
\frac{u-v}{2}\right)  \frac{\phi}{2},\\
T\phi &  =i\partial_{y}\phi+\left(  \bar{k}_{n}\cos\frac{u+v}{2}-k_{n}%
\cos\frac{u-v}{2}\right)  \frac{\phi}{2},\\
M\eta &  =-i\partial_{y}\eta+\left(  \bar{k}_{n}\frac{u+v}{2}-k_{n}\cos
\frac{u-v}{2}\right)  \frac{\eta}{2},\\
N\eta &  =-\partial_{x}\eta-\left(  \bar{k}_{n}\cos\frac{u+v}{2}+k_{n}%
\cos\frac{u-v}{2}\right)  \frac{\eta}{2}.
\end{align*}

To simplify the notation, we write $\tilde{f}_{n}$ as $f$, and $\tilde{g}_{n}$
as $g.$ Explicitly, using the formulas of $u$ and $v,$ we find that $L\phi$ is
equal to
\[
\partial_{x}\phi-\left(  \bar{k}_{n}\left(  \frac{\left(  f\gamma
-g\tau\right)  ^{2}}{\left(  f^{2}+g^{2}\right)  \left(  \gamma^{2}+\tau
^{2}\right)  }-1\right)  +k_{n}\left(  \frac{\left(  f\gamma+g\tau\right)
^{2}}{\left(  f^{2}+g^{2}\right)  \left(  \gamma^{2}+\tau^{2}\right)
}-1\right)  \right)  \phi.
\]
The analysis of this operator is complicated by the fact that the function
$\frac{\tau}{\gamma}$ is purely imaginary, hence $\gamma^{2}+\tau^{2}$ will be
equal to zero at some points of $\mathbb{R}^{2}.$ We define this singular set
as
\[
S:=\left\{  \left(  x,y\right)  :\gamma^{2}+\tau^{2}=0\right\}  .
\]
Note that the asymptotic behavior of $\gamma$ and $\tau$ are determined by
some explicit exponential functions. It follows that for each fixed $y,$ there
are only finitely many points in $S.$ Now we denote
\[
\Gamma\left(  x,y\right)  :=\bar{k}_{n}\frac{2\left(  f\gamma-g\tau\right)
^{2}}{\left(  f^{2}+g^{2}\right)  \left(  \gamma^{2}+\tau^{2}\right)  }.
\]
Then $\Gamma$ is singular at $S$ and
\begin{align*}
L\phi &  =\partial_{x}\phi-\operatorname{Re}\left(  \Gamma-\bar{k}_{n}\right)
\phi,\\
T\phi &  =i\partial_{y}\phi+i\operatorname{Im}\left(  \Gamma-\bar{k}%
_{n}\right)  \phi.
\end{align*}
Rotating the axis if necessary, we can assume $p_{j}\neq0,$ for any $j,$ and
$p_{n}>0.$

\begin{lemma}
For any fixed $y\in\mathbb{R},$
\[
\Gamma\left(  x,y\right)  \rightarrow0\text{ as }x\rightarrow\pm\infty.
\]

\end{lemma}

\begin{proof}
This follows directly from analyzing the main order of $f,g$ and $\gamma,\tau
$, as $\left\vert x\right\vert \rightarrow+\infty.$ Indeed, $\Gamma$ decays
exponentially fast as infinity.
\end{proof}

We define the function%
\[
\xi\left(  x,y\right)  :=\exp\left(  -x\operatorname{Re}\bar{k}_{n}%
+y\operatorname{Im}\bar{k}_{n}+\int_{-\infty}^{x}\operatorname{Re}\left(
\Gamma\left(  s,y\right)  \right)  ds\right)  .
\]
Then formally $L\xi=0,$ with $\xi\left(  x,y\right)  \rightarrow
e^{-x\operatorname{Re}\bar{k}_{n}+y\operatorname{Im}\bar{k}_{n}},$ as
$x\rightarrow-\infty.$ However, since $\Gamma$ has singularities, it is not
clear at this moment whether $\xi$ is well defined. Nevertheless, we will show
that $\xi$ is continuous.

We would like to analyze the singular set of $\Gamma$ away from the origin.

\begin{lemma}
\label{gamma}Let $\left(  x_{j},y_{j}\right)  $ be a sequence of points in $S$
such that $x_{j}^{2}+y_{j}^{2}\rightarrow+\infty,$ as $j\rightarrow+\infty.$
Then up to a subsequence, there is an index $j_{0}$ and sequence $A_{j}%
\in\mathbb{R},$ such that$,$
\[
\Gamma\left(  x_{j},y_{j}\right)  k_{j_{0}}\left(  p_{j_{0}}x_{j}+q_{j_{0}%
}y_{j}+A_{j}\right)  \rightarrow1,\text{ as }j\rightarrow+\infty.
\]

\end{lemma}

\begin{proof}
It will be convenient to multiply both $\gamma$ and $\tau$ by $\exp\left(
\frac{1}{2}\left(  \tilde{\eta}_{1}+...+\tilde{\eta}_{n-1}\right)  \right)  .$
Using the fact that $\left\vert \frac{\tau}{\gamma}\right\vert =1$ in $S,$ we
first infer that there exists an index $j_{0}$ and a universal constant $C$
such that $\left\vert \eta_{j_{0}}\right\vert \leq C$ for a subsequence of
$\left\{  \left(  x_{j},y_{j}\right)  \right\}  _{j=1}^{+\infty}.$(Otherwise,
$\left\vert \frac{\tau}{\gamma}\right\vert $ will be tending to $+\infty$ or
$0,$ depending on the parity of $n$).

We still denote this subsequence by $\left(  x_{j},y_{j}\right)  .$ Without
loss of generality, we may assume that as $j\rightarrow+\infty,$%
\begin{align*}
\eta_{m}  &  \rightarrow-\infty,\text{ for }m=1,...,j_{0}-1,\\
\eta_{m}  &  \rightarrow+\infty,\text{ for }m=j_{0}+1,...,n.
\end{align*}
We only consider the case that $n-j_{0}$ is odd. The other case is similar.

In view of the main order terms of $\tau$ and $\gamma,$ we get
\begin{equation}
\frac{\tau}{\gamma}\rightarrow\exp\left(  \tilde{\eta}_{j_{0}}\right)
\prod\limits_{j=1}^{j_{0}-1}\frac{k_{j}+k_{j_{0}}}{k_{j}-k_{j_{0}}}%
\prod\limits_{j=j_{0}+1}^{n-1}\frac{k_{j_{0}}-k_{j}}{k_{j_{0}}+k_{j}}.
\label{tau}%
\end{equation}
On the other hand, along this sequence $\left(  x_{j},y_{j}\right)  ,$%
\[
\frac{g}{f}\rightarrow\exp\left(  -\tilde{\eta}_{j_{0}}\right)  \prod
\limits_{j=1}^{j_{0}-1}\frac{k_{j}-k_{j_{0}}}{k_{j}+k_{j_{0}}}\prod
\limits_{j=j_{0}+1}^{n}\frac{k_{j_{0}}+k_{j}}{k_{j_{0}}-k_{j}}.
\]
Recall that $\gamma^{2}+\tau^{2}=1$ at $\left(  x_{j},y_{j}\right)  .$ Hence
\begin{equation}
\frac{g^{2}}{f^{2}}\rightarrow-\left(  \frac{k_{j_{0}}+k_{n}}{k_{j_{0}}-k_{n}%
}\right)  ^{2}. \label{fg}%
\end{equation}
Now we compute
\begin{align*}
\Gamma &  =\bar{k}_{n}\frac{2\left(  f\gamma-g\tau\right)  ^{2}}{\left(
f^{2}+g^{2}\right)  \left(  \gamma^{2}+\tau^{2}\right)  }\\
&  =2\bar{k}_{n}\frac{\left(  1-\frac{g}{f}\frac{\tau}{\gamma}\right)  ^{2}%
}{\left(  1+\frac{g^{2}}{f^{2}}\right)  \left(  1+\frac{\tau^{2}}{\gamma^{2}%
}\right)  }.
\end{align*}
Then by $\left(  \ref{tau}\right)  $ and $\left(  \ref{fg}\right)  ,$ as
$j\rightarrow+\infty,$
\[
\Gamma\left(  x_{j},y_{j}\right)  \left(  1+\frac{\tau^{2}}{\gamma^{2}%
}\right)  \rightarrow2\bar{k}_{n}\frac{\left(  1-\frac{k_{j_{0}}+k_{n}%
}{k_{j_{0}}-k_{n}}\right)  ^{2}}{1-\left(  \frac{k_{j_{0}}+k_{n}}{k_{j_{0}%
}-k_{n}}\right)  ^{2}}=-2\bar{k}_{j_{0}}.
\]
This then leads to the assertion of the lemma.
\end{proof}

By $\left(  \ref{tau}\right)  ,$ away from the origin, the singular set $S$
consists of finitely many components, each of them is asymptotic to a straight line.

\begin{lemma}
\label{time}Let $T_{1}:=L\phi-M\eta,$ $T_{2}:=T\phi-N\eta.$ Suppose that
$\Delta\eta=\eta\cos u$ and $T_{1}=0.$ Then
\begin{equation}
\partial_{x}T_{2}=\left(  \frac{\bar{k}_{n}}{2}\cos\frac{u+v}{2}+\frac{k_{n}%
}{2}\cos\frac{u-v}{2}\right)  T_{2}. \label{T2}%
\end{equation}

\end{lemma}

\begin{proof}
Let $\beta=\bar{k}_{n}.$ Consider the system%
\[
\left\{
\begin{array}
[c]{l}%
-\partial_{x}u-i\partial_{y}v+\beta\sin\frac{u+v}{2}+\beta^{-1}\sin\frac
{u-v}{2}=0\\
-i\partial_{y}u-\partial_{x}v-\beta\sin\frac{u+v}{2}+\beta^{-1}\sin\frac
{u-v}{2}=0
\end{array}
\right.
\]
Denoting the right hand side of the first equation by $A_{1}$, and that of the
second equation by $A_{2},$ we have%
\begin{align*}
\partial_{x}A_{2}-i\partial_{y}A_{1}  &  =-\Delta v-\beta\cos\frac{u+v}%
{2}\left(  \frac{\partial_{x}u+\partial_{x}v}{2}\right)  +\beta^{-1}\cos
\frac{u-v}{2}\left(  \partial_{x}u-\partial_{x}v\right) \\
&  -\beta i\cos\frac{u+v}{2}\left(  \frac{\partial_{y}u+\partial_{y}v}%
{2}\right)  -\beta^{-1}i\cos\frac{u-v}{2}\frac{\partial_{y}u-\partial_{y}v}%
{2}\\
&  =-\Delta v-\frac{\beta}{2}\cos\frac{u+v}{2}\left(  \partial_{x}%
u+\partial_{x}v+i\left(  \partial_{y}u+\partial_{y}v\right)  \right) \\
&  +\frac{\beta^{-1}}{2}\cos\frac{u-v}{2}\left(  \partial_{x}u-\partial
_{x}v-i\left(  \partial_{y}u-\partial_{y}v\right)  \right) \\
&  =-\Delta v-\frac{\beta}{2}\cos\frac{u+v}{2}\left(  2\beta^{-1}\sin
\frac{u-v}{2}-A_{1}-A_{2}\right) \\
&  +\frac{\beta^{-1}}{2}\cos\frac{u-v}{2}\left(  2\beta\sin\frac{u+v}{2}%
-A_{1}+A_{2}\right) \\
&  =-\Delta v+\sin v+A_{1}\left(  \frac{\beta}{2}\cos\frac{u+v}{2}-\frac
{\beta^{-1}}{2}\cos\frac{u-v}{2}\right) \\
&  +A_{2}\left(  \frac{\beta}{2}\cos\frac{u+v}{2}+\frac{\beta^{-1}}{2}%
\cos\frac{u-v}{2}\right)  .
\end{align*}
Differentiating this equation in $u,v$, we get the desired $\left(
\ref{T2}\right)  .$
\end{proof}

\begin{proposition}
$\xi$ is well defined in $\mathbb{R}^{2}.$ Near each point $\left(
x_{0},y_{0}\right)  \in S,$ $\xi\left(  x,y\right)  =O\left(  x-x_{0}\right)
.$ Moreover, $T\xi=0$ in $\mathbb{R}^{2}.$
\end{proposition}

\begin{proof}
Let $\left(  x_{0},y_{0}\right)  \in S.$ First we consider the case that
$\left\vert y_{0}\right\vert $ is large. From Lemma \ref{gamma}, we infer that
near $x_{0},$ $\xi\left(  x,y_{0}\right)  =O\left(  \left\vert x-x_{0}%
\right\vert ^{\alpha}\right)  ,$ where $\alpha$ is close to $1.$ Hence $\xi$
is well defined for $\left\vert y\right\vert $ large, say $\left\vert
y\right\vert >C_{0}.$

We wish to show that in the region $\Omega_{1}:=\left\{  \left(  x,y\right)
:y>C_{0}\right\}  ,$ $T\xi=0.$ Let $y_{1}\in\lbrack C_{0},+\infty).$ Suppose
$S\cap\left\{  \left(  x,y_{1}\right)  :x\in\mathbb{R}\right\}  =\left\{
s_{1},...,s_{k}\right\}  ,$ where $s_{j}<s_{j+1}$ and they depends on $y_{1}.$
Let $x_{1}\in\left(  -\infty,s_{1}\right)  $ and $\rho$ be a function to be
determined. Consider the the problem
\begin{equation}
\left\{
\begin{array}
[c]{l}%
T\left(  \rho\xi\right)  =0,\text{ for }x=x_{1},\\
\text{ }\rho\left(  y_{1}\right)  =1.
\end{array}
\right.  \label{t}%
\end{equation}
Note that
\[
T\left(  \rho\xi\right)  =\rho^{\prime}\xi+\left(  \partial_{y}\xi-\left(
\operatorname{Im}\Gamma-\operatorname{Im}\bar{k}_{n}\right)  \xi\right)
\rho.
\]
Therefore the problem $\left(  \ref{t}\right)  $ is an ODE for $\rho$ and has
a unique solution $\rho$ in a small interval $\left(  y_{1}-\delta
,y_{1}+\delta\right)  $. Using Lemma \ref{time}, we know that $T\left(
\rho\xi\right)  =0$ in the strip $\Omega_{2}:=\left(  -\infty,x_{1}%
+\delta\right)  \times\left(  y_{1}-\delta,y_{1}+\delta\right)  .$ Hence in
this region,
\[
\rho^{\prime}\xi+\left(  \partial_{y}\xi-\left(  \operatorname{Im}%
\Gamma-\operatorname{Im}\bar{k}_{n}\right)  \xi\right)  \rho=0.
\]
Dividing both sides by $\xi$ and letting $x\rightarrow-\infty,$ we find that
$\rho^{\prime}\left(  y\right)  =0$, thus $\rho\left(  y\right)  =1.$ This
implies that the function $\xi$ solves $T\left(  \xi\right)  =0$ in
$\Omega_{2}.$

Next we proceed to analyze the asymptotic behavior of $\xi$ near the left most
singularity $s_{1},$ when $\partial_{x}\Gamma^{-1}$ is nonzero at $s_{1}$(This
holds when $y$ is large). Assume $\xi$ has the form $\beta\left(  y\right)
e^{y\operatorname{Im}\bar{k}_{n}}\left(  s_{1}\left(  y\right)  -x\right)
^{\alpha\left(  y\right)  },\alpha,\beta$ are unknown functions, and
$\beta\neq0,\alpha$ is close to $1.$ We call $\alpha$ the vanishing order of
$\xi$. Then
\begin{align}
T\left(  \xi\right)  e^{-y\operatorname{Im}\bar{k}_{n}}  &  =\partial_{y}%
\xi-\left(  \operatorname{Im}\Gamma-\operatorname{Im}\bar{k}_{n}\right)
\xi\nonumber\\
&  =\beta^{\prime}\left(  y\right)  \left(  s_{1}-x\right)  ^{\alpha\left(
y\right)  }+\beta\left(  y\right)  \alpha\left(  y\right)  \left(
s_{1}-x\right)  ^{\alpha\left(  y\right)  -1}s_{1}^{\prime}\nonumber\\
&  +\beta\left(  y\right)  \left(  s_{1}-x\right)  ^{\alpha\left(  y\right)
}\ln\left(  s_{1}-x\right)  \alpha^{\prime}\left(  y\right) \nonumber\\
&  -\beta\left(  y\right)  \left(  s_{1}-x\right)  ^{\alpha\left(  y\right)
}\operatorname{Im}\Gamma\nonumber\\
&  =0. \label{g0}%
\end{align}
Here $s_{1}$ is evaluated at $y.$ In the last identity, dividing both sides
with $\left(  s_{1}-x\right)  ^{\alpha\left(  y\right)  -1}$ and letting
$x\rightarrow s_{1},$ we obtain%
\begin{equation}
\alpha\left(  y\right)  s_{1}^{\prime}-\left[  \left(  s_{1}-x\right)
\operatorname{Im}\Gamma\right]  |_{x=s_{1}}=0. \label{ga}%
\end{equation}
Using the real analyticity of $\Gamma^{-1},$ we can expand $\Gamma$ around
$x=s_{1}.$ Dividing $\left(  \ref{g0}\right)  $ by $\left(  s_{1}\left(
y\right)  -x\right)  ^{\alpha\left(  y\right)  }$ and using $\left(
\ref{ga}\right)  ,$ we find that $\alpha^{\prime}\left(  y\right)  =0.$ Hence
$\alpha$ is a constant. When $y\rightarrow+\infty,$ we know from Lemma
\ref{gamma} that $\alpha\left(  y\right)  \rightarrow1.$ It follows that
$\alpha$ is identically equal to one along each unbounded connected component
of $S$ containing $s_{1}.$

In principle, $S$ could have bounded connected components(We don't know
whether this can actually happen). Assume now that $s_{1}$ is belonging to a
bounded component $B_{1}$. Using the previous argument, one can first prove
that the vanishing order $\alpha$ of $\xi$ in $B_{1}$ is constant. We now show
that $\alpha$ is actually positive. Indeed, observe that the functions
$f,g,\gamma,\tau$ contain parameters $k_{1},...,k_{n}.$ We can deform these
parameters to the situation that all $k_{j}$ are close to $k_{n}.$ For a
generic deformation, the vanishing order of the corresponding functions $\xi
$(also depends on $k_{j}$) will not change sign(Note that we don't know
whether the vanish order will change along this deformation). But in the case
that $k_{j}$ are all close to $k_{n}$, bounded components of singular set will
not appear and thus the vanishing order are equal to one, thus positive. This
tells us that $\alpha>0.$

Now we have proved that $\xi$ solves $T\xi=0$ for $x<$ $s_{1}\left(  y\right)
.$ To prove that $T\xi=0$ for any $x_{1}\in\left(  s_{1}\left(  y\right)
,s_{2}\left(  y\right)  \right)  ,$ we still consider the function $\phi
:=\rho\left(  y\right)  \xi\left(  x,y\right)  ,$ with $\rho\left(  y\right)
=1.$ One can solve the problem $T\phi=0$ for $x=x_{1}.$ Due to the asymptotic
behavior of $\phi$ at $x\rightarrow s_{1}\left(  y\right)  ,$ $\rho^{\prime
}=0$ and hence $\rho=1.$ Arguing in this way, we finally prove that $T\xi=0$
in $\mathbb{R}^{2}.$ The proof is thus completed.
\end{proof}

With the vanishing order of $\xi$ being understood, we proceed to solve the
system $\left(  \ref{LB}\right)  ,$ with $\eta$ being a bounded kernel of the
linearized elliptic sine-Gordon equation%
\begin{equation}
\Delta\eta+\eta\cos u=0. \label{linear}%
\end{equation}

For each fixed $y,$ the first inhomogeneous equation in $\left(
\ref{LB}\right)  $ has a solution of the form
\begin{equation}
\phi\left(  x,y\right)  =\xi\left(  x,y\right)  \int_{-\infty}^{x}\xi
^{-1}\left(  s,y\right)  M\eta ds. \label{fi}%
\end{equation}

\begin{lemma}
\label{A}Let $\eta$ be a bounded solution of $\left(  \ref{linear}\right)  .$
The function $\phi$ defined by $\left(  \ref{fi}\right)  $ satisfies system
$\left(  \ref{LB}\right)  .$ As a consequence, $\phi$ is a kernel of the
linearized elliptic sine-Gordon equation at $v,$ that is,
\begin{equation}
\Delta\phi+\phi\cos v=0. \label{jo}%
\end{equation}

\end{lemma}

\begin{proof}
By the definition of $\xi,$ it is always nonnegative. By multiplying $\xi$ by
$+1$ or $-1$ in different connected components of $\mathbb{R}^{2}\backslash
S,$ we get a $C^{1}$ function $\xi^{\ast}$ solving $L\xi^{\ast}=T\xi^{\ast
}=0.$ We wish to show that $\phi$ solves $T\phi=N\eta.$ Let $\left(
x_{1},y_{1}\right)  \in\mathbb{R}^{2}\backslash S.$ Consider the function
\[
\Phi\left(  x,y\right)  :=\phi\left(  x,y\right)  +\rho\left(  y\right)
\xi^{\ast}\left(  x,y\right)  ,
\]
where $\rho$ satisfies
\[
\left\{
\begin{array}
[c]{l}%
\rho^{\prime}\left(  y\right)  \xi^{\ast}\left(  x,y\right)  =-T\phi
+N\eta,\text{ for }x=x_{1},y\in\left(  y_{1}-\delta,y_{1}+\delta\right) \\
\rho\left(  y_{1}\right)  =0.
\end{array}
\right.
\]
Then $T\Phi=0$ for $x=x_{1},y\in\left(  y_{1}-\delta,y_{1}+\delta\right)  .$
Using Lemma \ref{time}, for $y\in\left(  y_{1}-\delta,y_{1}+\delta\right)  ,$
$\Phi$ satisfies the system
\[
\left\{
\begin{array}
[c]{c}%
L\Phi=M\eta,\\
T\Phi=N\eta.
\end{array}
\right.
\]
Hence
\[
\rho^{\prime}\left(  y\right)  \xi^{\ast}\left(  x,y\right)  =-T\phi
+N\eta,\text{ for }y\in\left(  y_{1}-\delta,y_{1}+\delta\right)  \text{.}%
\]
For each fixed $y,$ sending $x$ to $-\infty$ in the above equation, we get
$\rho^{\prime}\left(  y\right)  =0.$ Hence $\rho=0$ and $\Phi$ satisfies
system $\left(  \ref{LB}\right)  .$ It then follows from the linearization of
the B\"{a}cklund transformation that $\phi$ satisfies $\left(  \ref{jo}%
\right)  .$ The proof is completed.
\end{proof}

Now we are ready to prove the nondegeneracy theorem.

\begin{proof}
[Proof of Theorem \ref{Main}]Let us fixed a $2n$-end solution $u=U_{n}$ of
$\left(  \ref{esg}\right)  .$ Suppose $\eta$ is nontrivial bounded kernel of
the linearized operator. Note that in the definition of $U_{n},$ there are
$2n$ real parameters $\operatorname{Re}k_{j},\eta_{j}^{0},j=1,...,n.$
Differentiating with respect to these parameters in the elliptic sine-Gordon
equation, we obtain $2n$ linearly independent solutions of the equation
$\left(  \ref{linear}\right)  ,$ denoting them by $\zeta_{1},...,\zeta_{2n}.$
By adding suitable linear combinations of $\zeta_{j},j=1,...,2n,$ if
necessary, we can assume that $\eta\left(  x,y\right)  $ decays to zero
exponentially fast, as $x\rightarrow-\infty$. Applying Lemma \ref{A}, we get a
corresponding kernel $\phi_{n-1}$ of the linearized operator at the function
$4\arctan\frac{\tau}{\gamma},$ which can be regarded as a $n-1$-soliton type
solution of elliptic sinh-Gordon equation having singularities. Moreover,
$\phi_{n-1}$ is bounded and decays to zero as $x\rightarrow-\infty.$

Now similarly as before, $4\arctan\frac{\tau}{\gamma}$ is the B\"{a}cklund
transformation of an $n-2$-soliton type solution, which will be denoted by
$4\arctan\frac{\tau_{n-2}}{\gamma_{n-2}}.$ Repeating this procedure, we may
consider the B\"{a}cklund transformation between $4\arctan\frac{\tau_{j}%
}{\gamma_{j}}$ and $4\arctan\frac{\tau_{j-1}}{\gamma_{j-1}},$ where
$4\arctan\frac{\tau_{j}}{\gamma_{j}}$ is a $j$-soliton, and $4\arctan
\frac{\tau_{0}}{\gamma_{0}}=0.$ Linearizing these B\"{a}cklund transformation
and solving them similarly as in Lemma \ref{A}(One also need to be careful
about the point singularities in these systems), we finally get a bounded
kernel $\varphi_{0}$ of the operator
\[
\Delta\varphi_{0}-\varphi_{0}=0.
\]
Moreover, we may assume that $\varphi_{0}$ is decaying to zero as
$x\rightarrow-\infty.$ Hence $\varphi_{0}=0.$ This together with an analysis
of the reverse B\"{a}cklund transformation ultimately tell us that $\eta=0.$
This finishes the proof.
\end{proof}

\section{Inverse scattering transform and the classification of multiple-end
solutions}

The rest of the paper will be devoted to the proof of Theorem \ref{Main2}. We
consider the elliptic sine-Gordon equation in the form
\begin{equation}
\partial_{x}^{2}u+\partial_{y}^{2}u=\sin u,\text{ }0<u<2\pi. \label{SG2}%
\end{equation}
Multiple-end solutions of $\left(  \ref{SG}\right)  $ are corresponding to
those solutions of $\left(  \ref{SG2}\right)  $ whose $\pi$ level sets are
asymptotic to finitely many half straight lines at infinity. Along these half
lines, the solutions resemble the one dimensional heteroclinic solution
$\arctan e^{s}$ in the transverse direction. In this section, we will classify
these solutions using the inverse scattering transform framework developed in
\cite{Gut}.

Let $\sigma_{i},i=1,2,3$ be the Pauli spin matrices, that is,
\[
\sigma_{1}=\left[
\begin{array}
[c]{cc}%
0 & 1\\
1 & 0
\end{array}
\right]  ,\sigma_{2}=\left[
\begin{array}
[c]{cc}%
0 & -i\\
i & 0
\end{array}
\right]  ,\sigma_{3}=\left[
\begin{array}
[c]{cc}%
1 & 0\\
0 & -1
\end{array}
\right]  .
\]
Let $\lambda$ be a complex spectral parameter. The equation $\left(
\ref{SG2}\right)  $ has a Lax pair
\begin{align}
\Phi_{x}  &  =\frac{1}{2}\left(  \left(  \frac{i\lambda}{2}+\frac{\cos
u}{2i\lambda}\right)  \sigma_{3}-\frac{i}{2}\left(  u_{x}+iu_{y}\right)
\sigma_{2}-\frac{i\sin u}{2\lambda}\sigma_{1}\right)  \Phi,\label{Lax1}\\
\Phi_{y}  &  =\frac{1}{2}\left(  -\left(  \frac{\lambda}{2}+\frac{\cos
u}{2\lambda}\right)  \sigma_{3}+\frac{1}{2}\left(  u_{x}-iu_{y}\right)
\sigma_{2}-\frac{\sin u}{2\lambda}\sigma_{1}\right)  \Phi. \label{Lax2}%
\end{align}
Let $k\left(  \lambda\right)  =\lambda-\frac{1}{\lambda}.$ Note that due to
the asymptotic behavior of $u,$ as $x\rightarrow\pm\infty,$ the coefficient
matrix of the righthand side of $\left(  \ref{Lax1}\right)  $ tends to the
constant matrix $\frac{i}{4}k\sigma_{3}.$ Let $\Phi_{\pm}$ be the solution of
$\left(  \ref{Lax1}\right)  $ such that $\Phi_{\pm}\left(  x,y\right)
\sim\exp\left(  \frac{i}{4}k\sigma_{3}x\right)  ,$ as $x\rightarrow\pm\infty.$
Note that $\Phi_{+}$ and $\Phi_{-}$ are solutions of the same ODE system. For
$\lambda\in\mathbb{R},$ they are related by
\[
\Phi_{+}\left(  x,y,\lambda\right)  =\Phi_{-}\left(  x,y,\lambda\right)
\left[
\begin{array}
[c]{cc}%
a\left(  \lambda\right)  & b\left(  \lambda\right) \\
-b\left(  -\lambda\right)  & a\left(  -\lambda\right)
\end{array}
\right]  .
\]
The functions $a\left(  \lambda,y\right)  ,b\left(  \lambda,y\right)  $ are
called the scattering data, which is a priori depending on $y$ and the
spectral parameter $\lambda.$ In equation $\left(  \ref{Lax2}\right)  ,$
sending $x\rightarrow-\infty,$ we know that they obey the following evolution
laws along the $y$ direction:%
\begin{align*}
a\left(  \lambda,y\right)   &  =a\left(  \lambda,0\right)  ,\\
b\left(  \lambda,y\right)   &  =b\left(  \lambda,0\right)  \exp\left(
-\frac{1}{4}\left(  \lambda+\lambda^{-1}\right)  y\right)  .
\end{align*}
Since $u$ is a smooth bounded solution which looks like the gluing of finitely
many one dimensional heteroclinic solution as $\left\vert y\right\vert
\rightarrow+\infty,$ we must have $b\left(  \lambda,y\right)  =0$ for nonzero
$\lambda\in\mathbb{R}$(otherwise, it blows up exponentially fast).

Since $u-\pi$ is a multiple-end solution of $\left(  \ref{SG}\right)  ,$ there
exists a choice of parameters $p_{j},q_{j},\eta_{j}^{0}$ such that the zero
level set of the corresponding solution $U_{n}-\pi$ has the same asymptotic
lines as $u-\pi,$ as $y\rightarrow+\infty.$ We denote the $a$ part of the
scattering data of $U_{n}$ by $a_{U_{n}}\left(  \lambda,y\right)  ,$ and that
of $u$ by $a_{u}\left(  \lambda,y\right)  .$ Then since $U_{n}$ and $u$ have
the same asymptotic behavior as $y\rightarrow+\infty$, we must have
\[
a_{U_{n}}\left(  \lambda,y\right)  =a_{u}\left(  \lambda,y\right)  .
\]
The potentials $U_{n}$ and $u$ in the Lax pair can be recovered by the inverse
scattering procedure(See equations (14), (15) in \cite{Gut}). It follows that
$U_{n}$ and $u$ are two reflection-less potential having the same scattering
data. Therefore $u=U_{n}.$

\bigskip

\section{Morse index of the multiple-end solutions}

In this section, we shall compute the Morse index of the multiple-end
solutions through a deformation argument. We have proved that the multiple-end
solutions $U_{n}-\pi$ are the only $2n$-end solutions. Therefore, the space
$M_{n}$ of $2n$-end solutions endowed with the natural topology defined in
\cite{Michal} has exactly one connected component. We now know that they are
$L^{\infty}$ nondegenerate. Hence for fixed $n,$ the Morse index of all the
solutions in $M_{n}$ are same.

\begin{proposition}
\label{Morse}The Morse index of $U_{n}-\pi$ is equal to $n\left(  n-1\right)
/2.$
\end{proposition}

\begin{proof}
First of all, we observe that by the result of \cite{Gui2}, when $n=2,$ the
Morse index of $U_{n}-\pi$ is equal to $1.$ We have developed in \cite{K3} an
end-to-end construction scheme for multiple-end solutions of the Allen-Cahn
equation. Roughly speaking, for each $n\geq2,$ we can glue $n\left(
n-1\right)  /2$ four-end solutions together by matching their ends.
Geometrically, the centers of these four-end solutions are far away from each
other. The zero level set of the solution looks like a desingularization of
the intersection of $n$ lines, where the intersection points are far away from
each other.

It will be suffice for us to show that the Morse index of the solutions $u$
obtained from the end-to-end construction have Morse index $n\left(
n-1\right)  /2.$ We use $z_{1}\left(  u\right)  ,...,z_{n\left(  n-1\right)
/2}\left(  u\right)  $ to denote the centers of the corresponding four-end
solutions $g_{1}\left(  u\right)  ,...,g_{n\left(  n-1\right)  /2}\left(
u\right)  ,$ and use $\eta_{j}\left(  u\right)  $ with $\left\Vert \eta
_{j}\right\Vert _{L^{\infty}}=1$ to denote a choice of the negative
eigenfunctions of the operator $-\Delta+\cos g_{j}.$ Since $z_{j}$ are far
away from each other and $\eta_{j}$ decays exponentially fast at infinity, we
can show that the Morse index of $u$ is at least $n\left(  n-1\right)  /2,$
and each $\eta_{j}$ can be perturbed into a true eigenfunction $\eta_{j}%
^{\ast}$ with negative eigenvalue.

We now show that the Morse index of $u$ is at most $n\left(  n-1\right)  /2,$
if the distances between any two centers for the four-end solutions are large
enough. We will argue by contradiction and assume to the contrary that there
exists a sequence of solutions $u_{k}$ and a sequence of corresponding
negative eigenfunction $\phi_{k}$ of the operator $-\Delta-\cos u_{k}$, with
eigenvalue $\lambda_{k},$ such that $\phi_{k}$ is orthogonal to each $\eta
_{j}^{\ast}\left(  u_{k}\right)  ,$ $j=1,...,n\left(  n-1\right)  /2$. We
normalize it such that $\left\Vert \phi_{k}\right\Vert _{L^{\infty}}=1.$ We
consider two cases.

Case 1. There is a sequence of points $Z_{k}$ such that $\left\vert \phi
_{k}\left(  Z_{k}\right)  \right\vert >\frac{1}{2},$ and as $k\rightarrow
+\infty,$ $\min_{j}$dist$\left(  Z_{k},z_{j}\left(  u_{k}\right)  \right)
\rightarrow+\infty.$

Note that in this case, the distance of $Z_{k}$ to the zero level set of
$u_{k}$ has to be uniformly bounded, otherwise, $\phi_{k}$ will converge
around $Z_{k}$, to a nontrivial bounded solution $\Phi$ of the equation
\[
-\Delta\Phi+\Phi=0\text{ in }\mathbb{R}^{2},
\]
which is impossible. Since away from the centers $z_{j}\left(  u_{k}\right)
$, $u_{k}$ looks like the one dimensional heteroclinic solution, we can show
that $\lambda_{k}\rightarrow0$ as $k\rightarrow+\infty.$ Recall that for each
$u_{k},$ the operator $-\Delta-\cos u_{k}$ has $2n$ linearly independent
kernels $\zeta_{k,1},...,\zeta_{k,2n},$ which grow at most linearly at
infinity. Analyzing the asymptotic behavior of $\phi_{k}$ more closely, we
know that actually $\phi_{k}$ is close to certain linear combination of
$\zeta_{k,j},j=1,...,2n.$ But this contradicts with the fact that $\phi_{k}$
is orthogonal to $\zeta_{k,j}.$

Case 2. $\phi_{k}\left(  z\right)  \rightarrow0$ as $\left\vert z-z_{j}\left(
u_{k}\right)  \right\vert \rightarrow+\infty,$ for each $j,$ uniformly in $k.$

In this case, we still choose $Z_{k}$ such that $\phi_{k}\left(  Z_{k}\right)
=\frac{1}{2}.$ Then dist$\left(  Z_{k},z_{j_{k}}\left(  u_{k}\right)  \right)
\leq C$ for some index $j_{k}.$ Consider the function $\varphi_{k}\left(
z\right)  :=\phi\left(  z-Z_{k}\right)  .$ Then $\varphi_{k}$ converges to a
decaying eigenfunction $\varphi_{\infty}$ of a four-end solution. The
corresponding eigenvalue has to be negative, since the linearized operator of
the four-end solution has no decaying kernel. This contradicts with the
assumption that $\phi_{k}$ is orthogonal to $\eta_{j}^{\ast}\left(
u_{k}\right)  ,j=1,...,2n.$

In conclusion, the Morse index of $u$ has to be $n\left(  n-1\right)  /2$ if
the distance between those $z_{j}\left(  u\right)  $ are large.
\end{proof}

We remark that in \cite{Manuel}, multiple-end solutions with almost parallel
ends have been constructed. The zero level set of these solutions are close to
solutions of the $n$-component Toda system%
\begin{equation}
\left\{
\begin{array}
[c]{l}%
q_{1}^{\prime\prime}=-e^{q_{1}-q_{2}},\\
q_{j}^{\prime\prime}=e^{q_{j}-q_{j+1}}-e^{q_{j-1}-q_{j}},j=2,...,n-1,\\
q_{n}^{\prime\prime}=e^{q_{n-1}-q_{n}}.
\end{array}
\right.  \label{Toda}%
\end{equation}
The Morse index of these solutions is equal to the Morse index of the Toda
system. Since $\left(  \ref{Toda}\right)  $ is a system of ODE, its solutions
are automatically $L^{\infty}$ nondegenerate. A corollary of Proposition
\ref{Morse} is that each solution of $\left(  \ref{Toda}\right)  $ has Morse
index $\frac{n\left(  n-1\right)  }{2}.$

\end{document}